\documentclass[a4paper,12pt]{amsart}
\usepackage[utf8]{inputenc}
\usepackage{fullpage}
\usepackage[english]{babel}
\usepackage{enumerate}
\usepackage[usenames,dvipsnames]{xcolor}
\usepackage[colorlinks=true,linkcolor=Blue,citecolor=Green]{hyperref}
\usepackage{bookmark}
\usepackage{amsmath,amsthm,amsfonts,amssymb}
\usepackage{graphicx}
\newcommand\invisiblesection[1]{%
  \refstepcounter{section}%
  \addcontentsline{toc}{section}{\protect\numberline{\thesection}#1}%
  \sectionmark{#1}}

\newtheorem{theorem}{Theorem}[section]
\newtheorem{proposition}[theorem]{Proposition}
\newtheorem{cor}[theorem]{Corollary}
\newtheorem{lemma}[theorem]{Lemma}
\theoremstyle{definition}
\newtheorem{definition}{Definition}
\newtheorem{remark}[theorem]{Remark}

\newcommand{\Si}{\mathcal{S}}
\newcommand{\C}{\mathcal{C}}

\renewcommand{\Re}{{\mathbb R}}
\newcommand{\Red}{\Re^d}
\newcommand{\Ci}{\mathcal{C}}
\newcommand{\Ei}{E}
\newcommand{\Elli}{\mathcal{E}}
\newcommand{\HH}{\mathcal{H}}
\newcommand{\Ball}{\mathbf{B}^d}
\newcommand{\emphdef}[1]{\emph{\textbf{#1}}}
\newcommand{\qhellyvol}{c^{d^2}d^{-5d^2/2+d}}
\newcommand{\noshow}[1]{}
\newcommand{\Href}[2]{\hyperref[#2]{#1~\ref{#2}}}

\newcommand{\vol}[1]{\mathrm{vol}\left(#1\right)}
\newcommand{\conv}[1]{\mathrm{conv}\left(#1\right)}
\newcommand{\st}{\;:\;}
\newcommand{\mybinom}[2]{\left(\begin{array}{@{}c@{}}#1\\#2\end{array}\right)}

\title{Quantitative Fractional Helly and \texorpdfstring{\lowercase{$(p,q)$}}{\lowercase{(p,q)}}-Theorems}

\author{Attila Jung}
\address{A.J.: Inst. of Mathematics, ELTE E\"otv\"os Lor\'and University, Budapest, Hungary}
\email{jungattila@gmail.com}

\author{M\'arton Nasz\'odi}
\address{M. N.: 
MTA-ELTE Lend\"ulet Combinatorial Geometry Research Group, Budapest, Hungary and
Dept. of Geometry, ELTE E\"otv\"os Lor\'and University, Budapest, Hungary}
\email{marton.naszodi@math.elte.hu}

\makeatletter
\@namedef{subjclassname@2020}{%
  \textup{2020} Mathematics Subject Classification}
\makeatother

\begin{document}
\begin{abstract}
We consider quantitative versions of Helly-type questions, that is, instead of finding 
a point in the intersection, we bound the volume of the intersection.
Our first main result is a quantitative version of 
the Fractional Helly Theorem of Katchalski and Liu, the second one is a quantitative version of the $(p,q)$-Theorem of Alon and Kleitman.
\end{abstract}

\maketitle

\section{Introduction}\label{sec:intro}

The study of quantitative versions of Helly-type questions was initiated by 
B\'ar\'any, Katchalski and Pach in \cite{BKP82}, where the \emphdef{Quantitative 
Volume Theorem} is shown, stating the following. \label{page:QVT}
\emph{Assume that the intersection of any $2d$ members of a finite family of convex sets in $\Red$ is of volume at least one. Then the volume of the intersection of all members of the family is of volume at least $c(d)$, a constant depending on $d$ only.} 

In \cite{BKP82}, it is proved that one can take $c(d)=d^{-2d^2}$ and conjectured that it should 
hold with $c(d)=d^{-cd}$ for an absolute constant $c>0$. It was confirmed  with 
$c(d)\approx d^{-2d}$ in \cite{MR3439267}, whose argument was refined by 
Brazitikos \cite{Bra17}, who showed that one may take $c(d)\approx d^{-3d/2}$. 
For more on quantitative Helly-type results, see the surveys 
\cite{HW18survey,DGFM19survey} and the recent papers \cite{RoSo17, MR3602856, 
SaXuSo}.

In the present note, we continue the study started in \cite{DFN2019} (where a 
quantitative variant of the Colorful Helly Theorem is shown) by proving variants 
of the Fractional Helly Theorem, the $(p,q)$-Theorem and related results.

\noshow{
The \emphdef{Colorful Helly Theorem}\label{page:chelly} found by 
Lov\'asz \cite{Lo74} (and with the first published proof by B\'ar\'any 
\cite{MR676720}) states the following.
\emph{If 
$\Ci_1,\dots, \Ci_{d+1}$ are finite families (color classes) of convex sets in 
$\Red$, such 
that for any colorful selection $C_1\in \Ci_1,\dots, C_{d+1}\in \Ci_{d+1}$, the 
intersection $\bigcap\limits_{i=1}^{d+1} C_i$ is non-empty, then for some $j$, 
the intersection $\bigcap\limits_{C\in \Ci_j} C$ is also non-empty.}

In \cite{DFN2019}, the following quantitative variant is shown.

\begin{theorem}[Quantitative Colorful Helly Theorem]\label{thm:colorfulqellipsoid}
Let $\C_1,\linebreak[0]\ldots,\C_{3d}$ be finite families of convex sets in 
$\Red$. Assume that for any colorful choice of $2d$ sets, $C_{i_k}\in \C_{i_k}$ 
for each $1\leq k\leq 2d$ with $1\leq i_1<\ldots<i_{2d}\leq 3d$, the 
intersection $\bigcap\limits_{k=1}^{2d} C_{i_k}$ is of volume at least one. 

\noindent Then, there is an $i$ with $1\leq i \leq 3d$ such that 
$\vol{\bigcap\limits_{C\in \C_i}C}\geq c(d)$ with some 
$c(d)>0$.
\end{theorem}
}

The \emphdef{Fractional Helly Theorem}  
due to Katchalski and Liu \cite{KL79} (see also \cite[Chapter 8]{MR1899299})
states the following. 
\label{page:fracHelly}
\emph{Fix a dimension $d$, and an $\alpha\in(0,1)$, and let $\C$ be a finite 
family of convex sets in $\Red$ with the property that among the subfamilies of 
$\C$ of size $d+1$, there are at least $\alpha \binom{|\C|}{d+1}$ for whom the 
intersection of the $d+1$ members is nonempty. Then, there is a subfamily 
$\C^\prime\subset\C$ of size $|\C^{\prime}|\geq\frac{\alpha}{d+1} |\C|$ such 
that the intersection of all members of $\C^\prime$ is nonempty.}

Our first main result is the following, a Quantitative Fractional Helly Theorem (QFH in short).

\begin{theorem}[QFH]\label{thm:QFHvolFew}
 For every dimension $d \geq 1$ and every $\alpha\in(0,1)$, there is a 
 $\beta\in(0,1)$ such that the following holds.

\noindent Let $\C$ be a finite family of convex sets in $\Red$. Assume that among all 
subfamilies of size $3d+1$, there are at least $\alpha \binom{|\C|}{3d+1}$ for 
whom the intersection of the $3d+1$ members is of volume at least one.

    \noindent Then, there is a subfamily $\C^\prime\subset\C$ of size at least 
$\beta |\C|$ such that the intersection of all members of $\C^\prime$ is of 
volume at least $d^{-cd^2}$ with a universal constant $c>0$.
\end{theorem}

\begin{remark}[Ellipsoids and volume]\label{rem:ellipsoidvsvolume}
A well known consequence of John's Theorem, is that the volume of any compact 
convex set $K$ in $\Red$ with non-empty interior is at most $d^d$ times larger than the 
volume of the largest volume ellipsoid contained in $K$. More precise bounds for 
this volume ratio are known (cf. \cite{Ba97}), but we will not need them. Thus, 
\emph{from this point on, we phrase our results in terms of the volume of ellipsoids 
contained in intersections}. Its benefit is that this is how in the proofs we 
actually ``find volume'': we find ellipsoids of large volume. 
\end{remark}

In this spirit, we re-state our result in terms of ellipsoids. It immediately yields Theorem~\ref{thm:QFHvolFew} by Remark~\ref{rem:ellipsoidvsvolume}.

\begin{theorem}[QFH in terms of ellipsoids]\label{thm:QFHfew}
 For every dimension $d \geq 1$ and every $\alpha\in(0,1)$, there is a $\beta\in(0,1)$ such that the following holds.

 \noindent Let $\C$ be a finite family of convex sets in $\Red$. Assume that among all subfamilies of size $3d+1$, there are at least $\alpha \binom{|\C|}{3d+1}$ for whom the intersection of the $3d+1$ members contains an ellipsoid of volume one.

\noindent Then, there is a subfamily $\C^\prime\subset\C$ of size at least 
$\beta |\C|$ such that the intersection of all members of $\C^\prime$ contains 
an ellipsoid of volume at least $\qhellyvol$, where $c$ is the universal 
constant from Theorem~\ref{lem:qhellyell}.
\end{theorem}

\begin{remark}
	The lower bound on $\beta$ obtained by our proof is $\frac{\alpha}{C^d}$ with a universal constant $C>1$. Note that in the case of the classical Fractional Helly Theorem, the sharp bound on $\beta$ shown in \cite{Ka84} by Kalai is $\beta = 1 - (1-\alpha)^\frac{1}{d+1}$.
\end{remark}

Holmsen \cite{holmsen19} recently showed that, in an abstract setting, a colorful Helly-type result implies a fractional one. In \cite{SaXuSo},
the authors combine Holmsen's result with a Quantitative Colorful Helly Theorem (\cite[Proposition~1.4]{DFN2019}, \cite[Corollary~1.0.3]{SaXuSo}) to obtain a quantitative fractional result. We note that this approach works only when the size of subfamilies considered is large, $\frac{d(d+3)}{2}$.

\begin{proposition}[QFH -- large subfamilies]\label{prp:QFHmany}
	For every dimension $d \geq 1$ and every $\alpha \in (0,1)$ the following holds.
	
	Let $\C$ be a finite family of convex bodies in $\Red$. Assume that among all subfamilies of size $\frac{d(d+3)}{2}$, there are at least $\alpha\binom{|\C|}{\frac{d(d+3)}{2}}$ for whom the intersection of the $\frac{d(d+3)}{2}$ members contains an ellipsoid of volume one.
	
	Then, there is a subfamily $\C' \subset \C$ of size at least $\beta|\C|$ such that the intersection of all members of $\C'$ contains an ellipsoid of volume one, where $\beta=\frac{2\alpha}{d(d+3)}$.
\end{proposition}

\begin{remark}
	In general, the value of $\beta$ we obtain is better than the one Holmsen's abstract result \cite{holmsen19} would yield, which is roughly $\left(\frac{\alpha}{3d^4}\right)^{\left(d/\sqrt{2}\right)^{d^2}}$. On the other hand, in \cite{holmsen19}, it is shown that $\beta \to 1$ as $\alpha \to 1$.
\end{remark}

Next, we turn to the $\mathbf{(p,q)}$-\textbf{Theorem}, a strong generalization of Helly's Theorem, conjectured by Hadwiger and Debrunner in 1957 \cite{hadwiger1957variante} and proved by Alon and Kleitman \cite{alon1992piercing} in 1992. It states that \textit{there is a function $H$ such that if we have a family $\C$ of convex sets in $\Red$, and among any $p$ of them there are $q$ with a nonempty intersection ($p \geq q \geq d+1$), then there is a set $T$ with $|T| \leq H(p,q,d)$ such that every member of $\C$ contains at least one point from $T$}.

Observe that for any $p\geq q\geq d+1$, we have that $H(p,q,d)\leq H(p,d+1,d)$. Thus, to show the existence of the function $H$, it is sufficient to show that $H(p,d+1,d)$ is finite for all $p\geq d+1$.

\begin{definition}[Quantitative $v$-transversal number]
	For a $v>0$, we say that a family $T$ of ellipsoid of volume $v$ is a 
	\emph{quantitative $v$-transversal} of a family $\C$ of 
	convex bodies in $\Red$, if every $C \in \C$ contains at least one ellipsoid from $T$. The quantitative $v$-transversal number $\tau = \tau (\C, v)$ for $\C$ is the smallest cardinality of a quantitative $v$-transversal.
	
	If there is a $C \in \C$ that contains no ellipsoid of volume $v$, then we set $\tau (\C, v)=\infty$.
\end{definition}

Our second main result follows. We phrase it without reference to $q$ in the definition of $H$, in fact, we fix $q$ at the smallest possible value 
for which we can prove the statement, $q=3d+1$.

\begin{theorem}[Quantitative \texorpdfstring{$(p,q)$}{(p,q)}-Theorem -- small lower bound on $p$]\label{thm:pqsmallp}
	For every dimension $d \geq 1$ and every positive integer $p \geq 3d+1$, there is an integer 
	$H = H(p,d)$ such that the following holds.
	
	Let $\C$ be a finite family of convex sets in $\Red$, each containing an ellipsoid of volume one, and assume that among any $p$ members of $\C$, there exist $3d+1$ whose intersection contains an ellipsoid of volume one.
	
	Then, $\C$ has quantitative $v$-transversal number at most $H(p,d)$ with $v=\qhellyvol$, where $c$ is the universal constant from Theorem~\ref{lem:qhellyell}.
\end{theorem}

The following variant is presented in \cite{SaXuSo}, where $p$ is larger than in Theorem~\ref{thm:pqsmallp}, it is $\frac{d(d+3)}{2}$ and, in return, we obtain a quantitative 1-transversal, that is, there is no loss in the volume of the ellipsoids.

\begin{proposition}[Quantitative \texorpdfstring{$(p,q)$}{(p,q)}-Theorem -- large lower bound on $p$, {\cite[Theorem~5.0.1]{SaXuSo}}]\label{prp:pqlargep}	
	For every dimension $d \geq 1$ and every positive integer $p \geq \frac{d(d+3)}{2}$, there is an integer 
	$H = H(p,d)$ such that the following holds.
	
	Let $\C$ be a finite family of convex sets in $\Red$, each containing an ellipsoid of volume one, and assume that among any $p$ members of $\C$, there exist $\frac{d(d+3)}{2}$ whose intersection contains an ellipsoid of volume one.
	
	Then, $\C$ has quantitative $1$-transversal number at most $H(p,d)$.
\end{proposition}

For results where the number of selected sets is much larger and, in return, one obtains a good approximation of the volume (and not just the volume of the largest ellipsoid contained in a set), see \cite{MessinaSoberon20} and \cite{MR3602856}.

The structure of the paper is the following. In Section~\ref{sec:prelim}, we introduce some of our tools, mostly from \cite{DFN2019}.
We prove Proposition~\ref{prp:QFHmany} in Section~\ref{sec:QFHmany}. Next, as a preparation for the $(p,q)$ results, in Section~\ref{sec:QEpsilonNet}, we prove a quantitative version of the \textbf{Selection Lemma} and the \textbf{Weak Epsilon Net Theorem} (see the statements of these two classical results in that section).
Next, in Section~\ref{sec:pqlargep}, we give a short direct proof of Proposition~\ref{prp:pqlargep}.
The arguments in these sections are modeled on proofs of the corresponding classical (ie., non-quantitative) results: the combinatorial core of those arguments are extracted, and then combined with geometric facts from \cite{DFN2019}, which guarantee uniqueness of certain ellipsoids contained in a convex body (see Lemmas~\ref{lem:lowest} and \ref{lem:dsquare}). 

In contrast, the proofs of Theorems~\ref{thm:QFHfew} and \ref{thm:pqsmallp} require further ideas, that are presented in Sections~\ref{sec:QFHfew} and \ref{sec:pqsmallp}. The difficulty in proving these results, where ``rough approximation'' is required, that is, the number of sets to be selected is $O(d)$ and not $O(d^2)$ (and, in return, there is a substantial loss of volume) is explained at the beginning of Sections~\ref{sec:pqsmallp}.

\section{Preliminaries}\label{sec:prelim}

We will rely on the following quantitative Helly theorem that is phrased in 
terms of the volume of an \emph{ellipsoid} contained in a convex body.

\begin{theorem}[Quantitative Helly Theorem]\label{lem:qhellyell}
Let $C_1,\ldots,C_n$ be convex sets in $\Red$. Assume that the intersection of 
any $2d$ of them contains an ellipsoid of volume at least one. Then 
$\bigcap\limits_{i=1}^n C_i$ contains an ellipsoid of volume at least 
$c^dd^{-3d/2}$ with an absolute constant $c>0$.
\end{theorem}

Theorem~\ref{lem:qhellyell} was stated in \cite{MR3439267} for volumes of 
intersections and not volumes of ellipsoids with the weaker bound $c^dd^{-2d}$, 
which was an improvement of the volume bound $d^{-2d^2}$ given by B\'ar\'any, 
Katchalski, and Pach \cite{BKP82}. The proof in \cite{MR3439267} clearly yields 
containment of ellipsoids as stated herein, and the argument was 
later refined by Brazitikos \cite{Bra17}, who obtained the bound 
$c^dd^{-3d/2}$ as stated above.

Another one of our key tools is the following observation from \cite{DFN2019}.

\begin{lemma}[{\cite[Lemma~3.2]{DFN2019}}]\label{lem:hellyklee}
Assume that the origin centered Euclidean unit ball, $\Ball$ is the largest 
volume ellipsoid contained in the convex set $C$ in  $\Red$. Let $\Ei$ be 
another ellipsoid in $C$ of volume at least $\delta\vol{\Ball}$ with 
$0<\delta<1$. Then there is a translate of $\frac{\delta}{d^{d-1}} \Ball$ which 
is contained in $\Ei$.
\end{lemma}

We state an immediate corollary of the Fractional Helly Theorem (see 
page~\ref{page:fracHelly}).

\begin{proposition}\label{prop:fhellyklee}	
	For every dimension $d \geq 1$ and every $\alpha \in (0,1)$, the following 
holds.
	
\noindent	Let $\C$ be a finite family of convex sets in $\Red$ and let $L$ be 
a convex set in $\Red$. Assume that among all subfamilies of size $d+1$, there 
are at least $\alpha \binom{|\C|}{d+1}$ for whom the intersection of the $d+1$ 
members contains a translate of $L$.
	
\noindent	Then, there is a subfamily $\C^\prime\subset\C$ of size at least 
$\frac{\alpha}{d+1} |\C|$ such that the intersection of all members of 
$\C^\prime$ contains a translate of $L$.	
\end{proposition}

\begin{proof}[Proof of Proposition~\ref{prop:fhellyklee}]
We use the following operation, the \emph{Minkowski difference} of two convex 
sets $A$ and $B$:
\[
A\sim B:=\bigcap_{b\in B} (A-b).
\]
It is easy to see that $A\sim B$ is the set of those vectors $t$ such that 
$B+t\subseteq A$.

Now, replace each convex set $C$ in $\C$ by $C\sim L$, and apply the Fractional 
Helly Theorem (see page~\ref{page:fracHelly}).
\end{proof}

The following definition and two lemmas introduce the unique lowest ellipsoid of volume one contained in a convex body.

\begin{definition}
	For an ellipsoid $\Ei$, we define its \emph{height} as the largest value of the 
	orthogonal projection of $\Ei$ on the last coordinate axis.
\end{definition}

\begin{lemma}[Uniqueness of Lowest Ellipsoid, {\cite[Lemma~2.5]{DFN2019}}]\label{lem:lowest}
	Let $C$ be a convex body, such that it contains an ellipsoid of volume one. 
Then there is a unique ellipsoid of volume one 
	such that every other ellipsoid of volume one in $C$ has larger 
	height. We call this ellipsoid the \emph{lowest ellipsoid} in $C$.
\end{lemma}

\begin{lemma}[Lowest ellipsoid determined by $O(d^2)$ members of an intersection {\cite[Lemma~3.1]{DFN2019}}]\label{lem:dsquare}
	Let $C_1,\dots, C_{n}$ be a finite family of convex bodies in $\Red$ whose intersection contains an ellipsoid of volume one. Then, there are $d(d+3)/2-1$ indices $i_1,\dots,i_{d(d+3)/2-1}\in[n]$ such that  
	$\bigcap\limits_{i=1}^{n}C_i$, and 
	$\bigcap\limits_{j=1}^{d(d+3)/2-1}C_{i_j}$ have the same unique lowest ellipsoid.
\end{lemma}

\section{Proof of Proposition~\ref{prp:QFHmany} -- QFH for large subfamilies}
\label{sec:QFHmany}

	The following argument follows closely the proof of Theorem~8.1.1. in \cite{MR1899299}, the only difference is the use of the unique lowest ellipsoid of a convex body (Lemma~\ref{lem:lowest}) instead of its lexicographic minimum.
	
	Let $\C = \{C_1, \ldots, C_n\}$.
	We call an index set $I \in \mybinom{[n]}{\frac{d(d+3)}{2}}$ \emph{good}, if the 
	corresponding intersection $\cap_{i \in I} C_i$ contains an ellipsoid of volume 
	at least one.
	We say that a $\left(\frac{d(d+3)}{2}-1\right)$-element subset $S \subset I$ of a good index set $I\in 
	\mybinom{[n]}{\frac{d(d+3)}{2}}$ is a \emph{seed} of $I$, if $\cap_{i \in I}C_i$ and $\cap_{i \in S}C_i$ have the same lowest ellipsoid. By Lemma~\ref{lem:dsquare}, all good index sets 
	have a seed.
	
	Since we have $\alpha\mybinom{n}{\frac{d(d+3)}{2}}$ good index sets and only 
	$\mybinom{n}{\frac{d(d+3)}{2}-1}$ possible seeds, there is a $\left(\frac{d(d+3)}{2}-1\right)$-tuple $S\in\mybinom{[n]}{\frac{d(d+3)}{2}-1}$ 
	which is the seed of at least 
	\[
	\frac{\alpha\mybinom{n}{\frac{d(d+3)}{2}}}{\mybinom{n}{\frac{d(d+3)}{2}-1}} = 
	\alpha\frac{n-\frac{d(d+3)}{2} + 1}{\frac{d(d+3)}{2}}
	\] good index sets. Every such good index set has the form $S \cup \{i\}$ for an $i$. So we have $\alpha\frac{n-\frac{d(d+3)}{2} + 1}{d(d+3)/2}$ convex bodies containing the lowest ellipsoid of $\cap_{i \in S}C_i$, plus the $(\frac{d(d+3)}{2} -1)$ bodies from $S$. Hence, the lowest ellipsoid of $\cap_{i \in S}C_i$ is contained in at least \[
	\alpha\frac{n+1-d(d+3)/2}{d(d+3)/2} + \frac{d(d+3)}{2} -1 \geq
	\frac{2\alpha n}{d(d+3)}
	\]
	convex bodies among the $C_i$, completing the proof of Proposition~\ref{prp:QFHmany}.

\section{Proof of Theorem~\ref{thm:QFHfew} -- QFH for small subfamilies}\label{sec:QFHfew}

Let $\C = \{C_1, \ldots, C_n\}$.
 	We call an index set $I \in \binom{[n]}{3d+1}$ \emph{good}, if the 
corresponding intersection $\cap_{i \in I} C_i$ contains an ellipsoid of volume 
at least one.
 	We say that a $2d$-element subset $S \subset I$ of a good index set $I\in 
\binom{[n]}{3d+1}$ is a \emph{seed} of $I$, if the volume of the John ellipsoid 
of $\cap_{i \in S}C_i$ is at most $c^{-d}d^{3d/2}$ times the volume of the John 
ellipsoid of $\cap_{i \in I}C_i$, where $c$ is the absolute constant from 
Theorem~\ref{lem:qhellyell}. By Theorem~\ref{lem:qhellyell}, all good index sets 
have a seed.
 	
 	Since we have $\alpha\binom{n}{3d+1}$ good index sets and only 
$\binom{n}{2d}$ possible seeds, there is a $(2d)$-tuple $S\in\binom{[n]}{2d}$ 
which is the seed of at least \[\frac{\alpha\binom{n}{3d+1}}{\binom{n}{2d}} \geq 
\gamma \binom{n}{d+1}\] good index sets. Here $\gamma$ depends on $\alpha$ and 
$d$, but not on $n$.
 	
Let $I_1, \ldots, I_{\gamma \binom{n}{d+1}}$ 
be good index sets whose seed is $S$. Denote the John ellipsoid 
of the intersection $\cap_{i \in S}C_i$ by $\mathcal{E}$ and the John ellipsoid 
of $\cap_{i \in I_j}C_i$ by $\mathcal{E}_j$. By Lemma~\ref{lem:hellyklee}, for 
every $j$, there is a $v_j \in \Red$ such that $c^dd^{-5d/2 + 1}\mathcal{E} + 
v_j \subseteq \mathcal{E}_j$.

 	Thus, we have shown that at least $\gamma\binom{n}{d+1}$ of the $(d+1)$-wise 
intersections contain a translate of $c^dd^{-5d/2 + 1}\mathcal{E}$. We can apply 
\Href{Proposition}{prop:fhellyklee} with $L=c^dd^{-5d/2 + 1}\mathcal{E}$, which 
implies that there are $\frac{\gamma}{d+1} n$ such $C_i$, that their 
intersection contains a translate of $c^dd^{-5d/2 + 1}\mathcal{E}$. And, since 
$\mathcal{E}$ has volume at least one, this ellipsoid has volume at least 
$\qhellyvol$, completing the proof of \Href{Theorem}{thm:QFHfew}.

\section{Roadmap to \texorpdfstring{$(p,q)$}{(p,q)}: Selection Lemma and Weak Epsilon Net}
\label{sec:QEpsilonNet}

In \cite{Alon02}, partly by extracting the combinatorial arguments presented in earlier works, it is shown in an abstract setting (that is, working with hypergraphs in general, and not specifically with convex sets in $\Red$) that a $(p,q)$-theorem may be obtained from a fractional Helly-type theorem in the following manner. First, combined with a Tverberg-type theorem, a fractional Helly-type theorem yields a selection lemma \cite[Proposition~11]{Alon02}, which in turn yields a weak epsilon-net theorem \cite[Theorem~9]{Alon02} using a greedy algorithm. On the other hand, if a hypergraph satisfies the $(p,q)$ condition, then the hypothesis of the fractional Helly-type theorem holds. It follows that the fractional transversal number of the hypergraph is bounded from above, which, combined with a weak epsilon-net, yields that its transversal number is bounded as well, which is what the $(p,q)$-theorem states. In the rest of the paper, we will mark where we follow this path, and where we do not.

In this section, we state the above listed classical (non-quantitative) results along with their quantitative analogs.

\textbf{Tverberg's Theorem} (see \cite{Tv66}, and for a simpler proof \cite{Tv81}) sates the following. \emph{For every dimension $d\geq1$ and integer $r \geq 1$, if $m \geq (r-1)(d+1)+1$ and $\{x_1,\dots,x_m\}$ is a set of points in $\Red$ of size $m$, then there is a partition $\cup_{i = 1}^r I_i = [m]$ such that $\cap_{i=1}^r \conv{\{x_j\st j \in I_i\}}$ is not empty.
}.

It is shown in \cite[Proposition~10]{Alon02} that in an abstract setting, a fractional Helly-type theorem yields a Tverberg-type theorem. However, the Tverberg number (the lower bound on $m$) obtained there is very large. Luckily, we have a quantitative Tverberg theorem dues to Sarkar, Xue and Sober\'on with a much better Tverberg number.

\begin{theorem}[Quantitative Tverberg Theorem {\cite[Theorem~4.1.2]{SaXuSo}}]\label{thm:QTverberg}
	For every dimension $d\geq1$ and integer $r \geq 1$, if $m \geq (r-1)\left(\frac{d(d+3)}{2}+1\right)+1$ and $\{\Ei_1, \ldots, \Ei_m\}$ is a multiset of ellipsoids of volume one, then there is a partition $\cup_{i = 1}^rI_i = [m]$ such that $\cap_{i=1}^r \conv{\{\Ei_j\st j \in I_i\}}$ contains an ellipsoid of volume one.
\end{theorem}

We will need the following form of the above result.

\begin{cor}[Quantitative Tverberg Theorem with equal parts]\label{thm:QTverbergEqual}
	For every dimension $d \geq 1$, if $a = \left(\frac{d(d+3)}{2}-1\right)\left(\frac{d(d+3)}{2}+1\right)+1 $, $b = \frac{d(d+3)}{2}$ and $\{\Ei_1, \ldots,  \Ei_{ab}\}$ is a multiset of ellipsoids of volume one, then there is a partition $\cup_{i = 1}^bI_i = [ab]$ such that $|I_1| = \ldots = |I_b| = a$ and $\cap_{i=1}^b \conv{\{\Ei_j\st j \in I_i\}}$ contains an ellipsoid of volume one.
\end{cor}

\begin{proof}[Proof of Corollary~\ref{thm:QTverbergEqual}]
	Pick any $a$ element subset $I \subset [ab]$ and use Theorem~\ref{thm:QTverberg} with $r = b$. It yields a partition $\cup_{i=1}^bI'_i = I$ such that $\cap_{i=1}^b \conv{\{\Ei_j\st j \in I'_i\}}$ contains an ellipsoid of volume one. Now complete the parts of the obtained partition from the remaining $(ab-a)$ indexes into $a$-element disjoint index sets $I_1 \supset I'_1, I_2 \supset I'_2, \ldots, I_b \supset I'_b$. Since $\conv{\{\Ei_j\st j \in I'_i\}} \subset \conv{\{\Ei_j\st j \in I_i\}}$, the partition $\cup_{i = 1}^bI_i = [ab]$ will have the desired properties.
\end{proof}

The \textbf{Selection Lemma} (the planar version first proved by Boros and F\"uredi \cite{BF84}, the general version due to B\'ar\'any \cite{MR676720}) sates the following. 
\emph{
For every dimension $d \geq 1$, there exists a $\lambda \in (0,1)$ with the following property. If $\{x_1,\dots,x_n\}$ is a multiset of points in $\Red$, then there is a subset $\mathcal{H} \subseteq \binom{[n]}{d+1}$ with $|\mathcal{H}| \geq \lambda \binom{n}{d+1}$ such that $\cap_{I \in \mathcal{H}}\conv{\{x_j\st j \in I\}}$ is not empty.
}.

\begin{lemma}[Quantitative Selection Lemma]\label{thm:QSelection}
	For every dimension $d \geq 1$, there exists an integer $a(d)$ and a real number $\lambda(d) \in (0,1)$ with the following property. If $n \geq a$ and $\{\Ei_1, \ldots, \Ei_n\}$ is a multiset of volume one ellipsoids in $\Red$, then there is a subset $\mathcal{H} \subseteq \binom{[n]}{a}$ with $|\mathcal{H}| \geq \lambda \binom{n}{a}$ such that $\cap_{I \in \mathcal{H}}\conv{\cup\{\Ei_j\st j \in I\}}$ contains an ellipsoid of volume one.
\end{lemma}

The proof follows closely the proof of \cite[Proposition~11]{Alon02}.

\begin{proof}[Proof of Lemma~\ref{thm:QSelection}]
	
	Let $a = \left(\frac{d(d+3)}{2}-1\right)\left(\frac{d(d+3)}{2}+1\right)+1 $ and $b = \frac{d(d+3)}{2}$ as in Corollary~\ref{thm:QTverbergEqual}.
	
	Let $\Si = \left\{\conv{\cup\{\Ei_i\st  i \in I\}}\st I \in \binom{[n]}{a}\right\}$. Our plan is to show that a positive fraction of $b$-tuples in $\Si$ has an intersection which contains an ellipsoid of volume one in order to apply Proposition~\ref{prp:QFHmany} to $\Si$. Let
	
	\begin{align*}
		T =
		\bigg\{\{I_1, \ldots, I_b\}\st I_i \in \binom{[n]}{a}, I_i \cap I_j = \emptyset \text{ for } i \neq j \text{ and } \\ \cap_{i = 1}^b\conv{\cup\{\Ei_j\st j \in I_i\}} \text{ contains an ellipsoid of volume } 1\bigg\}.
	\end{align*}
	
	Observe that $|T|$ is at least $\binom{n}{ab}$, since, according to Corollary~\ref{thm:QTverbergEqual}, for each $J \in \binom{[n]}{ab}$ there exists pairwise disjoint $I_1, \ldots, I_b \in \binom{J}{a}$ such that $\cap_{i = 1}^b\conv{\cup\{\Ei_j\st j \in I_i\}}$ contains an ellipsoid of volume one, and so each $J$ contributes a different $b$-tuple in $T$. Hence	
	\[
	|T| \geq
	\binom{n}{ab} \geq
	\left(\frac{n}{ab}\right)^{ab} \geq
	\frac{1}{(ab)^{(ab)}}\mybinom{\binom{n}{a}}{b},
	\]
	which means that we can apply Proposition~\ref{prp:QFHmany} to $\Si$ and conclude that a $\lambda(d)=\beta\left(\frac{1}{(ab)^{(ab)}}, d\right)$ fraction of the members of $\Si$ has an intersection that contains an ellipsoid with volume $1$. This completes the proof of Lemma~\ref{thm:QSelection}.
\end{proof}


The \textbf{Weak Epsilon Net Theorem}, proved by Alon, B\'ar\'any, F\"uredi and Kleitman \cite{ABFK92} sates the following. 
\emph{
	For every dimension $d \geq 1$ there exists a function $f: (0,1] \rightarrow \Re$ with the following property.
	For any $\varepsilon\in(0,1]$, if $\C$ is a finite family of convex bodies in $\Red$, and $w:\Red\rightarrow[0,1]$ is a weight function such that $\sum_{x\in C} w(x)\geq \varepsilon$ for all $C\in\C$, and $\sum_{x\in \Red}w(x)=1$,
	then there is a set $S\subset\Red$ such that each $C \in \C$ contains an element of $S$, and $S$ is of size at most $f(\varepsilon)$.
}


\begin{theorem}[Existence of quantitative weak $\varepsilon$-nets]\label{thm:weakepsilonnet}
	For every dimension $d \geq 1$ there exists a function $f: (0,1] \rightarrow \Re$ with the following property.
	
	For any $\varepsilon\in(0,1]$, if $\C$ is a finite family of convex bodies in $\Red$, and $\Elli$ is a finite family of volume one ellipsoids in $\Red$, and $w:\Elli\rightarrow[0,1]$ is a weight function on this family of ellipsoids such that $\sum_{E\in\Elli, E\subseteq C} w(E)\geq \varepsilon$ for all $c\in\C$, and $\sum_{E\in\Elli} w(E)=1$,
	then there is a family $S$ of ellipsoids of volume one such that each $C \in \C$ contains a member of $S$, and $S$ is of size at most $f(\varepsilon)$.
\end{theorem}

\begin{proof}[Sketch of the proof of Theorem~\ref{thm:weakepsilonnet}]
	The existence of quantitative weak epsilon nets follows from the Quantitative Selection Lemma (Lemma~\ref{thm:QSelection}), using a greedy algorithm. A very similar proof can be found in \cite[Theorem~10.4.2.]{MR1899299}.
\end{proof}

We note that by using Theorem~\ref{thm:QTverberg}, we obtain a smaller $f(\varepsilon)$ than we would if we used \cite[Theorem~9]{Alon02}.

\section{Proof of Proposition~\ref{prp:pqlargep} -- Quantitative \texorpdfstring{$(p,q)$}{(p,q)}-Theorem with a large bound on \texorpdfstring{$p$}{p}}
\label{sec:pqlargep}

Let $\C$ be a finite family of convex bodies in $\Red$ as in Proposition~\ref{prp:pqlargep}. We will represent $\C$ as a finite hypergraph. For each subfamily of $\C$, take the lowest volume one ellipsoid contained in the intersection of that subfamily, if there is such an ellipsoid. This way, we have a finite family of volume one ellipsoids. This family of ellipsoids, denote it by $\Elli$, will be the vertex set of our hypergraph. For each $C\in\C$, consider the subfamily of ellipsoids in $\Ei$ that are contained in $C$. The edges of the hypergraph will be subfamilies of ellipsoids obtained this way. We denote this hypergraph by $(\Elli,\HH)$.

By \cite[Theorem~8]{Alon02} and Proposition~\ref{prp:QFHmany}, the fractional transversal number (see definition therein, or in \cite[Section~10]{MR1899299}) of $(\Elli,\HH)$ is bounded from above by some $T>0$ that depends on $d$ only.

Our Quantitative Weak $\varepsilon$-Net Theorem, Theorem~\ref{thm:weakepsilonnet}, applied with $\varepsilon=1/T$ now completes the proof of Proposition~\ref{prp:pqlargep}.

In summary, the proof of Proposition~\ref{prp:pqlargep} mainly follows the classical line of reasoning that yields the $(p,q)$-Theorem. As we will see in the next section, to obtain Theorem~\ref{thm:QFHfew}, some other ideas are required.

\section{Proof of Theorem~\ref{thm:pqsmallp} -- Quantitative \texorpdfstring{$(p,q)$}{(p,q)}-Theorem with a small bound on \texorpdfstring{$p$}{p}}\label{sec:pqsmallp}

First, we discuss why the same argument as in the previous section cannot be repeated. The main idea in Section~\ref{sec:pqlargep} was that we considered a hypergraph representing our convex sets, and studied properties of this hypergraph. In the setting of Theorem~\ref{thm:pqsmallp}, however, there are \emph{two} hypergraphs: we obtain one if we consider ellipsoids of volume one in our convex sets, and we obtain another, if ellipsoids of volume $v$ are considered (with $v=c^{d^2}d^{-5d^2/2+d}$). Thus, some additional care is required.

\begin{definition}[Quantitative fractional $v$-transversal number]
	For a subset $S \subseteq \Red$, let $E_v(S)$ be the set of volume $v$ ellipsoids contained in $S$.
	Let $\C$ be a family of subsets of $\Red$ and let $\varphi: E_v(\Red) \rightarrow [0,1]$ be a function that attains only finitely many nonzero values. We say that $\varphi$ is a \emph{quantitative fractional $v$-transversal} for $\C$, if $\sum_{\Ei \in E_v(C)} \varphi(\Ei) \geq 1$ for all $C \in \C$.
	The \emph{quantitative fractional transversal number} of $\C$ is the infimum of $\sum_{\Ei \in E_v(\Red)}\varphi(x)$ over all quantitative fractional transversals $\varphi$ of $\C$.
\end{definition}

\begin{lemma}[Boundedness of quantitative fractional $v$-transversal number]\label{lemma:boundedvtransversal}
	For every $d$ and $p \geq 3d+1$ there exists a $h = h(p,d)>0$ with the following property.
	
	Let $\C$ be a finite family of convex sets in $\Red$, each containing an ellipsoid of volume one, and assume that among any $p$ members of $\C$, there exist $3d+1$ whose intersection contains an ellipsoid of volume one.
	
	Then $\C$ has quantitative fractional $v$-transversal number less than $h$ with $v=c^{d^2}d^{-5d^2/2+d}$, where $c$ is the universal constant from Theorem~\ref{lem:qhellyell}.
\end{lemma}

\begin{proof}
	We are going to use linear programming duality as in \cite{alon1992piercing}, but first we need the definition of quantitative fractional $v$-matching numbers.

	\begin{definition}
		Let $\C$ be a finite family of convex sets in $\Red$ and let $m: \C \rightarrow [0,1]$ a function. We say that $m$ is a \emph{quantitative fractional $v$-matching} for $\C$, if for every ellipsoid $\Ei \subset \Red$ with volume $v$, the sum $\sum_{\Ei \subset C \in \C}m(C)$ is at most $1$. The \emph{quantitative fractional $v$-matching number} of $\C$ is the supremum of $\sum_{C \in \C}m(C)$ over all quantitative fractional $v$-matchings of $\C$.
	\end{definition}
	
    Let $v = \qhellyvol$.
	Note that if we consider each member $C$ of $\C$ as the collection of all volume $v$ ellipsoids that it contains, then we obtain a hypergraph with finitely many edges, but an infinite vertex set. However, we can replace this infinite vertex set by a finite one, just like at the beginning of Section~\ref{sec:pqlargep}, and obtain another hypergraph with the same (fractional) transversal and matching numbers as those of the original hypergraph. Now, we may apply the duality of linear programming to see that the quantitative fractional $v$-matching number and the quantitative fractional $v$-transversal number are equal for any family of convex sets $\C$. We denote it by $\nu_v^*(C)$. 
	
	We know also, that there is an optimal quantitative fractional $v$-matching taking only rational values. Let $m$ be such a quantitative fractional $v$-matching and suppose, that $m(C) = \frac{\tilde{m}(C)}{D}$ for every $C$, where $\tilde{m}(C)$ and $D$ are integers.
	
	Let $\tilde{\C} = \{C_1, \ldots, C_N\}$ be the multiset that contains $\tilde{m}(C)$ copies of each $C \in \C$. Taking some sets with multiplicities does not change the quantitative fractional matching number, so $\nu_v^*(\tilde{\C}) = \nu_v^*(\C)$.  We know that among any $p$ members of $\C$, there are $3d+1$ whose intersection contains an ellipsoid with volume $1$. So among any $\tilde{p} = 3d(p-1) + 1$ members of $\tilde{\C}$ there are $3d+1$ whose intersection contains an ellipsoid with volume $1$, because the $\tilde{p}$ element multiset from $\tilde{\C}$ either contains $3d+1$ copies fo the same set, or $p$ different sets from $\C$.
	
	For every $I \in \binom{[N]}{\tilde{p}}$ there is a subset $J \subset I$ with $|J| = 3d+1$ and $\cap_{j \in J} C_j$ containing an ellipsoid of volume one. So there are at least
	\[
	\frac{\binom{N}{\tilde{p}}}{\binom{N-3d+1}{\tilde{p}-3d+1}} \geq
	\alpha \binom{N}{3d+1}
	\]
	$3d+1$-tuples from $\C$ whose intersection contains an ellipsoid of volume one. We can apply Theorem~\ref{thm:QFHfew} and conclude that there are $\beta N$ sets from $\tilde{\C}$ whose intersection contains an ellipsoid with volume $v$. On the other hand, no ellipsoid of volume $v$ can be in more than $\frac{N}{\nu_1^*(\C)}$ of the sets from $\tilde{\C}$, hence $\nu_v^*(\C) \leq \frac{1}{\beta}$, completing the proof of Lemma~\ref{lemma:boundedvtransversal}.
\end{proof}

Now we are ready to prove Theorem~\ref{thm:pqsmallp}.

\begin{proof}[Proof of Theorem~\ref{thm:pqsmallp}.]
	If among any $p$ members of $\C$ there are $3d+1$ whose intersecton contains an ellipsoid of volume $1$, then from Lemma~\ref{lemma:boundedvtransversal} it follows, that $\C$ has quantitative fractional $v$-transversal number at most $h(p,d)$ with $v = \qhellyvol$. Let $\varphi$ be a quantitative fractional $v$-transversal of size $h(p,d)$, and for every ellipsoid $E$ with volume $v$ let $w(E) = \frac{\varphi(E)}{h(p,d)}$. We can use Theorem~\ref{thm:weakepsilonnet} with $w$ and $\C$ and conclude that there is a $\frac{1}{h}$-net $S$ with volume $v$ ellipsoids for $\C$ of size at most $f\left(\frac{1}{h}\right)$. Since for every $C \in \C$ the inequality $\sum_{E \subset C}w(E) \geq \frac{1}{h}$ holds, every $C \in \C$ contains at least one ellipsoid from $S$ and Theorem~\ref{thm:pqsmallp} follows with $H(p,d) = |S| = f\left(\frac{1}{h(p,d)}\right)$.
\end{proof}

\section{Concluding remarks}

The following questions are left open. First, can the Helly number $3d+1$ be replaced by $2d$ in Theorems~\ref{thm:QFHvolFew}, \ref{thm:QFHfew}, \ref{thm:pqsmallp}? It would be interesting to see such result, even if the volume bound on the ellipsoid is worse than the $d^{-cd^2}$, which we obtained.

Second, can the volume bound in the above mentioned theorems be replaced by $d^{-cd}$?

Finally, as mentioned in Section~\ref{sec:QFHmany}, according to \cite{holmsen19} by Holmsen, a colorful Helly type theorem yields a fractional Helly type theorem in the abstract setting, when one hypergraph is considered (see Section~\ref{sec:pqsmallp} for the explanation of one hypergraph vs. two). Can one give a quantitative analogue of this result? More precisely, if we make some assumptions on a pair of hypergraphs, then does a colorful Helly type theorem that holds for the pair of hypergraphs imply the corresponding fractional result for the same pair? That would mean that Theorem~\ref{thm:QFHvolFew} follows from \cite{DFN2019} by an abstract argument.

\invisiblesection{Acknowledgement}
\subsection*{Acknowledgement}
We thank Pablo Sober\'on for our discussions.
Both authors were supported by the grant EFOP-3.6.3-VEKOP-16-2017-00002.
MN was supported also by the J\'anos Bolyai Scholarship of the Hungarian Academy of Sciences as well as the National Research, Development and Innovation Fund (NRDI) grant K119670, as well as the \'UNKP-20-5 New National Excellence Program of the Ministry for Innovation and
Technology from the source of the NRDI.

\bibliographystyle{alpha}
\bibliography{biblio}

\end{document}